\documentclass{amsart}
\usepackage{graphicx}
\usepackage{caption}
\usepackage{tikz}
\usepackage{amsthm,amsmath,amssymb,amsfonts}
\usepackage{dsfont}
\usepackage{mathtools}
\usepackage{enumerate}
\usepackage{doi}
\usepackage{eucal}
\usepackage{mathrsfs}
\usepackage{stmaryrd}
\usepackage{verbatim}
\usepackage{manfnt}
\usepackage{array}
\usetikzlibrary{patterns,cd}
\usepackage{layout} 
\usepackage[all]{xy}

\begingroup
\catcode`\&=13
\gdef\pampmatrix{%
  \begingroup
  \let&=\amsamp
  \begin{pmatrix}%
}
\gdef\endpampmatrix{\end{pmatrix}\endgroup}
\endgroup

\newcommand{\R}{\mathbb{R}}
\newcommand{\F}{\mathbb{F}}
\newcommand{\Z}{\mathbb{Z}}
\newcommand{\C}{\mathbb{C}}
\newcommand{\N}{\mathbb{N}}
\newcommand{\Hom}{\operatorname{Hom}}
\newcommand{\co}{\colon\thinspace}
\newcommand{\Fin}{\mathcal{F}\mathrm{in}}
\newcommand{\Topp}{\operatorname{Top}}
\newcommand{\Gr}{\operatorname{Gr}}
\newcommand{\norm}[1]{\left\lVert#1\right\rVert}

\theoremstyle{plain}
\newtheorem{theorem}{Theorem}[section]
\newtheorem*{theorem*}{Theorem}
\newtheorem{proposition}[theorem]{Proposition}
\newtheorem{lemma}[theorem]{Lemma}
\newtheorem{corollary}[theorem]{Corollary}

\theoremstyle{definition}

\newtheorem{example}[theorem]{Example}
\newtheorem{construction}[theorem]{Construction}

\theoremstyle{remark}
\newtheorem{remark}[theorem]{Remark}

\usepackage[all]{xy}
\SelectTips{cm}{10}
\CompileMatrices

\DeclareMathAlphabet{\mathpzc}{OT1}{pzc}{m}{it}

\title[Commuting varieties and the rank filtration]{Commuting varieties and the rank filtration of topological $K$-theory}
\author[S.~Gritschacher]{Simon Gritschacher}
\address{Department of Mathematics\\
University of Munich, Munich, Germany }       
\email{simon.gritschacher@math.lmu.de}

\begin{document}

\maketitle

\begin{abstract}
We consider the space of $n$-tuples of pairwise commuting elements in the Lie algebra of $U(m)$. We relate its one-point compactification to the subquotients of certain rank filtrations of connective complex $K$-theory. We also describe the variant for connective real $K$-theory.
\end{abstract}


\section{Introduction}

The aim of this paper is to make a connection between varieties of commuting matrices and the rank filtration in topological $K$-theory. Let $ku$ be the connective complex $K$-theory spectrum and let $ku_n$ denote its $n$-th space, i.e., a connective $n$-fold delooping of $\Z\times BU$. By choosing a concrete model of $ku$, one can define for each $n\geq 1$ an ``unstable rank filtration" of $ku_n$, i.e., a filtration of the infinite loop space $ku_n$ by subspaces $R^s(ku_n)\subseteq ku_n$, $s\geq -1$. For example, for $n=1$ this is the familiar filtration of the stable unitary group $ku_1=U$ by the subspaces
\[
R^s(U)=\left\{ A\in U\mid \dim_{\C}(\ker(A-\mathrm{Id})^\perp)\leq s\right\}\, .
\]
The subquotients of this filtration are the Thom spaces
\[
R^s(U)/R^{s-1}(U)=\mathfrak{u}_s^+\wedge_{U(s)} EU(s)_+\,,
\]
where $\mathfrak{u}_s$ is the adjoint representation of the unitary group $U(s)$ (see e.g. \cite{Miller}). In this paper we describe more generally the subquotients of the rank filtration of $ku_n$ in terms of the commuting variety
\[
C_n(\mathfrak{u}_s)=\{(X_1,\dots,X_n)\in \mathfrak{u}_s^n\mid [X_i,X_j]=0\textnormal{ for all }1\leq i,j\leq n\}\, .
\]
We endow $C_n(\mathfrak{u}_s)$ with the Euclidean topology, and we let  $U(s)$ act on $C_n(\mathfrak{u}_s)$ through the adjoint representation. This extends to an action on the one-point compactification $C_n(\mathfrak{u}_s)^+$.

\begin{theorem} \label{thm:intro1}
For $n\geq 1$, the subquotients of the rank filtration of $ku_n$ are
\[
R^s(ku_n)/R^{s-1}(ku_n) \cong  C_n(\mathfrak{u}_s)^+\wedge_{U(s)} EU(s)_+\, .
\]
\end{theorem}

It should be noted that when $n=1$, a theorem of Miller \cite{Miller} asserts that the rank filtration gives a stable splitting $U_+\simeq \bigvee_{s\geq 0} \mathfrak{u}_s^+\wedge_{U(s)} EU(s)_+$. When $n=2$, one can show that the induced filtration on the cohomology of $ku_2=BU$ is precisely the monomial length filtration in the Chern classes. Because this filtration is not preserved by the action of the Steenrod algebra, it cannot split stably. Presumably, the filtration never splits when $n\geq 2$, but we do not show this in this paper.

The ``unstable" rank filtrations of $ku_n$ for various $n$ assemble into the stable rank filtration of the spectrum $ku$ (see \cite{Rognes} for the algebraic analogue, and \cite{AL,Hausmann} where it is termed the modified stable rank filtration of $ku$). The subquotients of the stable rank filtration were first described by Arone and Lesh \cite{AL} (see below). In this paper we arrive at the following description. Let $C_\infty(\mathfrak{su}_s)$ be the space of eventually zero infinite sequences $(X_1,X_2,X_3,\dots )$ of pairwise commuting traceless skew-hermitian $s\times s$ matrices. 

\begin{theorem} \label{thm:intro2}
The subquotients of the stable rank filtration of $ku$ are
\[
R^s(ku)/R^{s-1}(ku)\simeq \Sigma^\infty \left(C_{\infty}(\mathfrak{su}_s)^+\wedge_{U(s)}EU(s)_+\right)\,.
\]
\end{theorem}

In Proposition \ref{prop:universal1} we observe that the $U(s)$-space $C_\infty(\mathfrak{su}_s)^+$ is the suspension of a universal space for the collection of complete subgroups of $U(s)$ (see Section \ref{sec:homotopytype} for terminology). This implies that
\[
C_\infty(\mathfrak{su}_s)^+\simeq |\mathcal{L}_s|^\lozenge\,,
\]
where $|\mathcal{L}_s|^\lozenge$ denotes the unreduced suspension of the complex of proper orthogonal decompositions of $\C^s$ introduced by Arone \cite{AroneTop}. Thus, Theorem \ref{thm:intro2} is really the familiar result of Arone and Lesh \cite{AL}, who identified the subquotients of the stable rank filtration with the suspension spectra $\Sigma^\infty(|\mathcal{L}_s|^\lozenge\wedge_{U(s)}EU(s)_+)$.

As a by-product, by combining the equivalence $C_{\infty}(\mathfrak{su}_p)^+\simeq |\mathcal{L}_p|^\lozenge$ with the known cohomology of $|\mathcal{L}_p|$ from \cite{AroneTop} we obtain an independent computation of the mod-$p$ cohomology of the manifold of complete unordered flags in $\C^p$, obtained recently by Guerra and Jana \cite{JanaGuerra}. This is Corollary \ref{cor:flag}.

In Section \ref{sec:realvariants} we prove the real analogues of Theorems \ref{thm:intro1} and \ref{thm:intro2}. Let $C_n(\mathfrak{sym}_s)$ denote the space of $n$-tuples of pairwise commuting real symmetric $s\times s$-matrices. We then have, for example:
\begin{theorem} \label{thm:intro3}
For $n\geq 1$, the subquotients of the rank filtration of $ko_n$ are
\[
R^s(ko_n)/R^{s-1}(ko_n) \cong C_n(\mathfrak{sym}_s)^+\wedge_{O(s)}EO(s)_+\,.
\]
\end{theorem}

When $n=1$, then this is again the familiar rank filtration of the Lagrangian Grassmannian $ko_1=U/O$, which is known not to split stably at the prime $2$ (for a proof see \cite[Section 2]{Mitchell}).

We want to make clear that the filtrations considered here are well-studied, in particular by Arone and Lesh, and equivariantly by Hausmann and Ostermayr \cite{Hausmann}. For example, our space $C_n(\mathfrak{su}_s)^+$ turns out to be homeomorphic to the space $\overline{\mathcal{L}(s,S^n)}$ appearing in \cite[Section 3.1]{Hausmann}. However, the description using commuting varieties is new and elucidating this connection is the purpose of this article. As we showed in a joint paper with Adem and G{\'o}mez \cite{AGG}, the spaces $C_n(\mathfrak{u}_s)^+$ appear in the subquotients of a natural filtration of the representation space $\Hom(\Z^n,U(s))$. A motivation for the present work is the computation of the (equivariant) homology of $C_n(\mathfrak{u}_s)^+$, which we hope to pursue in a future paper.

In the following sections we chose to work with symmetric spectra (as opposed to orthogonal ones as in \cite{Hausmann}), because the action of the symmetric group $\Sigma_n$ on $C_n(\mathfrak{u}_s)^+$ is most easily described, and because it allows us to describe a model for the $K$-theory spectrum $ku$ which is based on spaces of commuting isometries (see Section \ref{sec:model}). Most of the results could be deduced from those in \cite{Hausmann} by proving once that their space $\overline{\mathcal{L}(s,S^n)}$ is homeomorphic to $C_n(\mathfrak{su}_s)^+$ (which is implicit in the proofs of Lemma \ref{lem:model} and Theorem \ref{thm:unstablerankfiltrationku}); but to keep our exposition self-contained we make the structure of the various spectra explicit in the commuting variety picture.

\section{Connective $K$-theory via commuting isometries} \label{sec:model}

We recall the construction of $ku$ as a commutative symmetric ring spectrum as detailed in  \cite[Section 6.3]{Global}. We will then describe in Construction \ref{cons:model} an isomorphic model which is based on commuting elements in unitary groups.

For two finite dimensional complex inner product spaces $V$ and $W$ let $L(V,W)$ denote the (compact) space of linear isometric embeddings $V\hookrightarrow W$ carrying the topology of the Stiefel manifold of $\dim(V)$-dimensional frames in $W$. If $W$ has countably infinite dimension, then $L(V,W)$ is defined as the colimit of $L(V,W')$ as $W'$ ranges over the finite dimensional subspaces of $W$. The unitary group $U(V)$ acts on $L(V,W)$ by precomposition, with quotient space
\[
\Gr_m(W)=L(\C^m,W)/U(m)\,,
\]
the Grassmannian of $m$-dimensional subspaces of $W$. If $W$ has countably infinite dimension, then $\Gr_m(W)$ is a model for the classifying space $BU(m)$.

Let $\Fin_\ast$ denote a skeleton of the category of finite pointed sets; its objects are the sets $\langle k\rangle:=\{0,1,\dots,k\}$ for $k\in \N$ with basepoint $0$ and its morphisms are the basepoint preserving maps. Let $\textnormal{Top}$ be the category of compactly generated weak Hausdorff spaces (as defined in \cite[Appendix A]{Global}). Let $\mathcal{U}$ be a finite or countably infinite dimensional complex inner product space. The $\Gamma$-space
\[
ku(-,\mathcal{U})\co \Fin_\ast\to \Topp
\]
assigns to the finite pointed set $\langle k\rangle$ the space
\[
ku(\langle k\rangle,\mathcal{U})=\bigsqcup_{n_1,\dots,n_k\geq 0} L(\C^{n_1}\oplus\cdots \oplus \C^{n_k},\mathcal{U})/(U(n_1)\times \cdots \times U(n_k))\,,
\]
i.e., the space of $k$-tuples of mutually orthogonal, finite dimensional subspaces of $\mathcal{U}$. For a based map $\alpha\co \langle k\rangle\to \langle l\rangle$ the map $ku(\alpha,\mathcal{U})$ is defined by sending $(V_i)_{i=1}^k$ to $(W_i)_{i=1}^l$ where $W_i=\bigoplus_{\alpha(j)=i} V_j$.

There is a standard way of extending $ku(-,\mathcal{U})$ to a functor on the category $\Topp_\ast$ of pointed spaces, namely by the coend
\[
ku(X,\mathcal{U})=\int^{\langle k\rangle\in \Fin_\ast} ku(\langle k\rangle,\mathcal{U})\times \mathrm{Map}_\ast(\langle k\rangle,X)\,.
\]
The elements of $ku(X,\mathcal{U})$ are equivalence classes of the form
\[
[(V_1,x_1),\dots,(V_k,x_k)]\,,
\]
where $(V_1,x_1),\dots,(V_k,x_k)$ is an unordered configuration of points $x_i\in X$ with labels $V_i\in \bigsqcup_{m\geq 0}\mathrm{Gr}_m(\mathcal{U})$ such that $V_i\perp V_j$ whenever $i\neq j$. The identifications that are made in the coend are listed, for example, in \cite[p. 837]{GH19}. The element $[(0,\ast)]$, where $\ast\in X$ is the basepoint, equips $ku(X,\mathcal{U})$ with a natural basepoint.

Suppose that $\C^n$ is equipped with the standard hermitian inner product. Let $\mathrm{Sym}(\C^n)$ denote the symmetric algebra of $\C^n$. There is a preferred inner product on $\mathrm{Sym}(\C^n)$, described in \cite[Proposition 6.3.8]{Global}, for which the isomorphism of algebras
\[
\psi_{n,m}\co \mathrm{Sym}(\C^n)\otimes_{\C} \mathrm{Sym}(\C^m)\xrightarrow{\cong} \mathrm{Sym}(\C^n\oplus \C^m)\,,
\]
induced by the two obvious inclusions $\C^n\to \C^n\oplus \C^m$ and $\C^m\to \C^n\oplus \C^m$, is an isometry. Moreover, this preferred inner product is natural with respect to linear isometric embeddings $\C^n\to \C^m$. In particular, the standard action of the symmetric group $\Sigma_n$ on $\C^n$ which permutes the coordinates extends to an action on $\mathrm{Sym}(\C^n)$ by isometries. If $\sigma\in \Sigma_n$, then we denote by $\sigma_\ast\co \mathrm{Sym}(\C^n)\to \mathrm{Sym}(\C^n)$ the induced isometry of $\mathrm{Sym}(\C^n)$.

The symmetric ring spectrum $ku$ is defined by setting
\[
ku_n=ku(S^n,\mathrm{Sym}(\C^n))\, .
\]
The action of the symmetric group $\Sigma_n$ on $ku_n$ is through both the action on $S^n$ (by viewing $S^n$ as the one-point compactification of the $n$-dimenensional standard representation of $\Sigma_n$) and the induced action on $\mathrm{Sym}(\C^n)$. Precisely, if $\sigma\in \Sigma_n$ and $[(V_1,x_1),\dots,(V_k,x_k)]\in ku_n$, then
\[
\sigma\cdot [(V_1,x_1),\dots,(V_k,x_k)]=[(\sigma_\ast(V_1),\sigma\cdot x_{1}),\cdots, (\sigma_\ast(V_k),\sigma\cdot x_{k})]\, .
\]
There are $(\Sigma_n\times \Sigma_m)$-equivariant multiplication maps
\[
\mu_{n,m}\co ku_n\wedge ku_m\to ku_{n+m}
\]
defined by
\[
[(V_1,x_1),\dots,(V_k,x_k)]\wedge [(W_1,y_1),\dots,(W_l,y_l)] \mapsto [(\psi_{n,m}(V_i\otimes W_j),x_i\wedge y_j)_{i,j}]\,,
\]
where on the right the indices range through $1\leq i\leq k$ and $1\leq j\leq l$. Here we are using implicitly the canonical homeomorphism $S^n\wedge S^m\cong S^{n+m}$. There is also a $\Sigma_n$-equivariant unit map
\[
\iota_n\co S^n\to ku_n
\]
defined by $x\mapsto [\C\cdot 1,x]$, where $\C\cdot 1\subseteq \mathrm{Sym}(\C^n)$ is the degree-$0$ summand. The $(\Sigma_n\times \Sigma_m)$-equivariant structure maps
\[
\sigma_{n,m}\co ku_n\wedge S^m\to ku_{n+m}
\]
are determined by $\sigma_{n,m}=\mu_{n,m}\circ (id \wedge \iota_m)$. This data makes $ku$ a commutative symmetric ring spectrum.

\begin{construction} \label{cons:model}
Let $\mathcal{U}$ be a finite or countably infinite dimensional complex inner product space and let
\[
U(\mathcal{U})=\mathrm{colim}_{\substack{V\subseteq \mathcal{U}\\ \dim_{\C}(V)< \infty}} U(V)
\]
be the ``stable" unitary group of $\mathcal{U}$. For $n\geq 1$ consider the space of $n$-tuples of pairwise commuting elements of $U(\mathcal{U})$,
\[
C_n(U(\mathcal{U}))=\{(A_1,\dots,A_n)\in U(\mathcal{U})^n\mid A_iA_j=A_jA_i\textnormal{ for all }1\leq i,j\leq n\}\, .
\]
On $C_n(U(\mathcal{U}))$ define the equivalence relation $(A_1,\dots,A_n)\sim(B_1,\dots,B_n)$ if
\[
\sum_{i=1}^n \ker(A_i-\mathrm{Id})=\sum_{i=1}^n \ker(B_i-\mathrm{Id})
\]
as subspaces of $\mathcal{U}$ and $A_i$ agrees with $B_i$ on the orthogonal complement of this subspace for all $i=1,\dots,n$. We denote the quotient space by
\[
\overline{\mathscr{C}}_n(\mathcal{U})=C_n(U(\mathcal{U}))/{\sim}.
\]
If $\mathcal{U}=\mathrm{Sym}(\C^n)$, then we write simply $\overline{\mathscr{C}}_n$. We define an action of $\Sigma_n$ on $\overline{\mathscr{C}}_n$ by
\[
\sigma\cdot [(A_1,\dots,A_n)]=[(\sigma_\ast A_{\sigma(1)} \sigma_\ast^{-1},\dots,\sigma_\ast A_{\sigma(n)} \sigma_\ast^{-1})]\, .
\]
It will be convenient to set $\overline{\mathscr{C}}_0:=S^0$.
\end{construction}

For example, $\overline{\mathscr{C}}_1(\mathcal{U})=U(\mathcal{U})$.

\begin{lemma} \label{lem:model}
For every $n\geq 0$, there is a $\Sigma_n$-equivariant based homeomorphism
\[
ku_n\cong \overline{\mathscr{C}}_n\, .
\]
\end{lemma}
\begin{proof}
When $n=0$ the unit map $\iota_0\co S^0\to ku_0$ is a homeomorphism and the statement holds trivially. Thus, assume that $n\geq 1$. We choose a homeomorphism of $S^1=\R^+$ with the unit sphere $S(\C)\subseteq \C$ by sending $t\mapsto (it-1)(it+1)^{-1}$ and the point at infinity to $1\in S(\C)$. We identify $S^n$ with the $n$-fold smash power $S(\C)\wedge \dots \wedge S(\C)$, so that any point $x_i\in S^n$ can be written as $x_{i1}\wedge \dots \wedge x_{in}$ with $x_{ij}\in S(\C)$. The desired homeomorphism is given by the map $\overline{\varphi}_n\co ku_n\to \overline{\mathscr{C}}_n$ defined by
\[
[(V_1,x_1),\dots,(V_k,x_k)] \mapsto [(A_1,\dots,A_n)]\,,
\]
where $A_j$ is determined uniquely by
\begin{itemize}
\item $A_j|_{V_i}=x_{ij} \mathrm{Id}_{V_i}$ for $i=1,\dots,k$
\item $A_j$ acts as the identity on the orthogonal complement of $\bigoplus_{i=1}^k V_i$.
\end{itemize}
This map is well-defined and a bijection by the spectral theorem. It is also easily checked that $\overline{\varphi}_n$ is $\Sigma_n$-equivariant.

To see that $\overline{\varphi}_n$ is a homeomorphism note that the evident maps $\overline{\mathscr{C}}_n(V)\to \overline{\mathscr{C}}_n$, where $V$ ranges over the finite dimensional subspaces of $\mathrm{Sym}(\C^n)$, induce a homeomorphism
\[
\mathrm{colim}_{V} \overline{\mathscr{C}}_n(V)\xrightarrow{\cong} \overline{\mathscr{C}}_n\, .
\]
Similarly, there is a homeomorphism $\mathrm{colim}_{V} ku(S^n,V)\cong ku_n$. It suffices to show that for each finite dimensional subspace $V\subseteq \mathrm{Sym}(\C^n)$ the restriction
\[
\overline{\varphi}_n|_V\co ku_n(S^n,V)\to \overline{\mathscr{C}}_n(V)
\]
is a homeomorphism. But this map is induced by a homeomorphism
\[
\varphi_n|_V\co ku((S^1)^n,V)\xrightarrow{\cong} C_n(U(V))
\]
upon passage to quotients. The definition of the map $\varphi_n|_V$ and the verification that it is a homeomorphism are in \cite[Proposition 2.4]{GH19}.
\end{proof}

\begin{remark}
To our knowledge, the homeomorphism in the case $n=1$ is originally due to Harris \cite{Harris}. It is also explained in \cite[Remark 6.3.3]{Global}.
\end{remark}

By virtue of the homeomorphism $ku_n\cong \overline{\mathscr{C}}_n$ the spaces $\overline{\mathscr{C}}_0,\overline{\mathscr{C}}_1,\overline{\mathscr{C}}_2,\dots$ can be assembled into a commutative symmetric ring spectrum isomorphic to $ku$. Note that for every $[(A_1,\dots,A_n)]\in \overline{\mathscr{C}}_n$ there is a unique largest subspace
\begin{equation}
\label{eq:minimalspace}
F(A_1,\dots,A_n)\subseteq \mathrm{Sym}(\C^n)
\end{equation}
restricted to which $A_i-\mathrm{Id}$ is non-singular for all $i=1,\dots,n$. This subspace is independent of the choice of representative of the equivalence class $[(A_1,\dots,A_n)]$.

The $(\Sigma_n\times \Sigma_m)$-equivariant multiplication maps
\[
\overline{\mathscr{C}}_n\wedge \overline{\mathscr{C}}_m\to  \overline{\mathscr{C}}_{n+m}
\]
are given by
\begin{equation*}
\begin{split}
& [(A_1,\dots,A_n)]\wedge [(B_1,\dots,B_m)] \mapsto [(C_1,\dots,C_{n+m})]
\end{split}
\end{equation*}
where
\[
C_i=\begin{cases} \psi_{n,m}(A_i|_{F(A_1,\dots,A_n)}\otimes \mathrm{Id}_{F(B_1,\dots,B_m)})\psi_{n,m}^{-1}  & \textnormal{if } 1\leq i \leq n\,,\\ 
  \psi_{n,m}(\mathrm{Id}_{F(A_1,\dots,A_n)}\otimes B_{i-n}|_{F(B_1,\dots,B_m)})\psi_{n,m}^{-1} & \textnormal{if } n+1\leq i \leq n+m\, . \end{cases}
\]
Here, $A_i|_{F(A_1,\dots,A_n)}\otimes \mathrm{Id}_{F(B_1,\dots,B_m)}$ and $\mathrm{Id}_{F(A_1,\dots,A_n)}\otimes B_{i-n}|_{F(B_1,\dots,B_m)}$ are viewed as isometries of $\mathrm{Sym}(\C^n)\otimes \mathrm{Sym}(\C^m)$ by extending them by the identity on the orthogonal complement of $F(A_1,\dots,A_n)\otimes F(B_1,\dots,B_m)$.

The $\Sigma_n$-equivariant unit map
\[
S^n\to \overline{\mathscr{C}}_n
\]
is given by
\[
x=x_1\wedge \dots \wedge x_n \mapsto [(x_1 \mathrm{Id}_{\C\cdot 1},\dots,x_n \mathrm{Id}_{\C\cdot 1})] \,,
\]
where $x_j \mathrm{Id}_{\C\cdot 1}$ is viewed as an isometry of $\mathrm{Sym}(\C^n)$ by extending it by the identity on the orthogonal complement of $\C\cdot 1$.

In particular, the $(\Sigma_n\times \Sigma_m)$-equivariant structure maps
\[
\overline{\mathscr{C}}_n\wedge S^m\to \overline{\mathscr{C}}_{n+m}
\]
are given by
\[
[(A_1,\dots,A_n)]\wedge (x_1\wedge \dots \wedge x_m)\mapsto [(B_1,\dots,B_{n+m})]\,,
\]
where
\[
B_i= \begin{cases} \psi_{n,m}(A_i|_{F(A_1,\dots,A_n)} \otimes \mathrm{Id}_{\C\cdot 1})\psi_{n,m}^{-1}  & \textnormal{if } 1\leq i \leq n\,,\\ 
  \psi_{n,m}(\mathrm{Id}_{F(A_1,\dots,A_n)}\otimes x_{i-n}\mathrm{Id}_{\C\cdot 1} )\psi_{n,m}^{-1} & \textnormal{if } n+1\leq i \leq n+m\, , \end{cases}
\]
and where again, $A_i|_{F(A_1,\dots,A_n)} \otimes \mathrm{Id}_{\C\cdot 1}$ and $\mathrm{Id}_{F(A_1,\dots,A_n)}\otimes x_{i-n}\mathrm{Id}_{\C\cdot 1} $ are regarded as isometries of $\mathrm{Sym}(\C^n)\otimes \mathrm{Sym}(\C^m)$ by extending them by the identity on the orthogonal complement of $F(A_1,\dots,A_n)\otimes \C\cdot 1$.

\section{The unstable rank filtration of $ku_n$}

We consider now, for fixed $n\in \N$, a filtration of the space $ku_n$ by subspaces and identify the associated graded in terms of commuting varieties. For an integer $s\geq -1$ we consider the subspace
\[
R^s(ku_n)=\left\{  [(V_1,x_1),\dots,(V_k,x_k)]\in ku_n \,\Big|\,
\sum_{\substack{i=1\\ x_i\neq \textnormal{ basepoint}}}^k \dim_\C(V_i)\leq s\right\}\, .
\]
This gives a filtration $R^0(ku_n)\subseteq R^1(ku_n)\subseteq \cdots$ with $ku_n=\bigcup_{s\geq 0} R^s(ku_n)$. Note that the $\Sigma_n$-action on $ku_n$ restricts to $R^s(ku_n)$. 

The following lemma guarantees that the strict subquotients $R^s(ku_n)/R^{s-1}(ku_n)$ model the homotopy cofibre of the inclusion $R^{s-1}(ku_n)\to R^s(ku_n)$.

\begin{lemma} \label{lem:cofibration}
For every $s\in \N$, the inclusion $R^{s-1}(ku_n)\to R^s(ku_n)$ is a cofibration.
\end{lemma}
\begin{proof}
This is proved (in greater generality) in \cite[Section 7.1]{Hausmann}.
\end{proof}

Let $C_n(\mathfrak{u}_s)^+$ denote the one-point compactification of
\[
C_n(\mathfrak{u}_s)=\{(X_1,\dots,X_n)\in \mathfrak{u}_s^n\mid [X_i,X_j]=0\textnormal{ for all }1\leq i,j\leq n\}\,,
\]
the space of $n$-tuples of pairwise commuting skew-Hermitian $s\times s$ matrices. The unitary group $U(s)$ acts on $C_n(\mathfrak{u}_s)$ by conjugation of each component of a tuple, and this action extends to one on $C_n(\mathfrak{u}_s)^+$.

On the half-smash product
\[
C_n(\mathfrak{u}_s)^+\wedge_{U(s)} L(\C^s,\mathrm{Sym}(\C^n))_+
\]
we define an action of the symmetric group $\Sigma_n$ by
\[
\sigma\cdot [(X_1,\dots,X_n)\wedge f] = [(X_{\sigma(1)},\dots,X_{\sigma(n)})\wedge \sigma_\ast\circ f]\, .
\]
The following proves Theorem \ref{thm:intro1} from the introduction (also cf. \cite[Cor. 3.4]{Hausmann}).

\begin{theorem} \label{thm:unstablerankfiltrationku}
For $s\in \N$ and $n\in \N_{>0}$, there is a $\Sigma_n$-equivariant homeomorphism
\[
R^s(ku_n)/R^{s-1}(ku_n)\cong C_n(\mathfrak{u}_s)^+\wedge_{U(s)} L(\C^s,\mathrm{Sym}(\C^n))_+\, .
\]
\end{theorem}
\begin{proof}
The statement holds trivially when $s=0$, so we will assume $s\geq 1$. Lemma \ref{lem:model} gives a $\Sigma_n$-equivariant homeomorphism $ku_n\cong \overline{\mathscr{C}}_n$ under which the subspace $R^s(ku_n)$ is mapped onto the subspace
\[
R^s(\overline{\mathscr{C}}_n)=\{[(A_1,\dots,A_n)]\in \overline{\mathscr{C}}_n\mid \dim_\C(F(A_1,\dots,A_n))\leq s\}
\]
(notation as in line (\ref{eq:minimalspace})). For a subspace $V\subseteq \mathrm{Sym}(\C^n)$ we define $R^s(\overline{\mathscr{C}}_n(V))\subseteq \overline{\mathscr{C}}_n(V)$ in the obvious fashion. Let $sC_n(U(s))\subseteq C_n(U(s))$ be the closed subspace consisting of those $n$-tuples $(A_1,\dots,A_n)$ such that $A_i-\mathrm{Id}$ is singular for at least one $i$. Then there is a homeomorphism
\[
R^s(\overline{\mathscr{C}}_n(V))-R^{s-1}(\overline{\mathscr{C}}_n(V)) \cong (C_n(U(s))-sC_n(U(s)))\times_{U(s)} L(\C^s,V)
\]
defined by
\begin{equation} \label{eq:homeo}
[(A_1,\dots,A_n)] \mapsto [(B_1,\dots,B_n),f]\,,
\end{equation}
where $f\co \C^s\to V$ is an isometric embedding with $\mathrm{im}(f)=F(A_1,\dots,A_n)$ and $B_i=(f')^{-1} A_i|_{F(A_1,\dots,A_n)} f'$ with $f'\co \C^s\to F(A_1,\dots,A_n)$ the co-restriction of $f$ onto its image.

Now we use the Cayley transform
\begin{equation} \label{eq:cayley}
\begin{split}
\psi_s\co  \mathfrak{u}_s & \to U(s) \\
 X & \mapsto (X-\mathrm{Id})(X+\mathrm{Id})^{-1}\, .
 \end{split}
\end{equation}
It is a $U(s)$-equivariant embedding of the space of $s\times s$ skew-Hermitian matrices onto the subspace of $U(s)$ consisting of all those matrices $A \in U(s)$ for which $A- \mathrm{Id}$ is non-singular. Applying the (inverse) Cayley transform to each entry of an $n$-tuple gives a $U(s)$-equivariant homeomorphism
\[
C_n(U(s))-sC_n(U(s)) \xrightarrow{\cong} C_n(\mathfrak{u}_s)\, .
\]
Together this gives a homeomorphism
\[
R^s(\overline{\mathscr{C}}_n(V))-R^{s-1}(\overline{\mathscr{C}}_n(V)) \cong C_n(\mathfrak{u}_s) \times_{U(s)} L(\C^s,V)\, .
\]
If $V$ is finite dimensional, then $L(\C^s,V)$ and $R^s(\overline{\mathscr{C}}_n(V))$ are compact. Since $R^{s-1}(\overline{\mathscr{C}}_n(V))$ is a closed subset of $R^{s}(\overline{\mathscr{C}}_n(V))$, one-point compactification gives a homeomorphism
\[
R^s(\overline{\mathscr{C}}_n(V))/R^{s-1}(\overline{\mathscr{C}}_n(V)) \cong C_n(\mathfrak{u}_s)^+ \wedge_{U(s)} L(\C^s,V)_+\, .
\]
This homeomorphism is natural with respect to isometric embeddings in $V$. By taking the colimit over all finite dimensional subspaces $V\subseteq \mathrm{Sym}(\C^n)$ we obtain a homeomorphism
\[
R^s(\overline{\mathscr{C}}_n)/R^{s-1}(\overline{\mathscr{C}}_n) \cong C_n(\mathfrak{u}_s)^+ \wedge_{U(s)} L(\C^s,\mathrm{Sym}(\C^n))_+\,.
\]
Away from the basepoint it is given by (\ref{eq:homeo}) (for $V=\mathrm{Sym}(\C^n)$). This map is easily seen to be $\Sigma_n$-equivariant. \end{proof}

\section{The stable rank filtration of $ku$} \label{sec:stablerankfiltrationku}

For each $s\in \N$, the structure maps of the spectrum $ku$ restrict to structure maps
\[
\sigma_{n,m}\co R^s(ku_n)\wedge S^m \to R^s(ku_{n+m})\,,
\]
defining a symmetric subspectrum $R^s(ku)$ of $ku$. By Lemma \ref{lem:cofibration}, the inclusion $R^{s-1}(ku)\to R^s(ku)$ is a levelwise cofibration. Hence, the symmetric spectrum obtained by forming the levelwise strict quotients is also the homotopy cofibre of the map of symmetric spectra $R^{s-1}(ku)\to R^s(ku)$. In this section we describe the subquotients $R^{s}(ku)/ R^{s-1}(ku)$ using commuting varieties.

Let $\mathfrak{su}_s$ denote the Lie algebra of the special unitary group $SU(s)$, i.e., the vector space of traceless skew-Hermitian $s\times s$ matrices. Let $C_n(\mathfrak{su}_s)$ denote the space of $n$-tuples of pairwise commuting elements in $\mathfrak{su}_s$. The group $U(s)$ acts on $C_n(\mathfrak{su}_s)$ componentwise by conjugation, inducing an action on the one-point compactification $C_n(\mathfrak{su}_s)^+$.

\begin{construction} \label{cons:quotientspectrum}
For every $s\in \N_{>0}$ we define a symmetric spectrum $C(s)$ by
\[
C(s)_n=(C_n(\mathfrak{su}_s)^+ \wedge S^n)\wedge_{U(s)}L(\C^s,\mathrm{Sym}(\C^n))_+\, ,
\]
where $U(s)$ acts trivially on $S^n$, and $\Sigma_n$ acts on $C(s)_n$ by
\[
\sigma\cdot [(X_1,\dots,X_n)\wedge x \wedge f]=[(X_{\sigma(1)},\dots,X_{\sigma(n)})\wedge \sigma\cdot x\wedge \sigma_\ast\circ f]\, .
\]
The structure map
\[
C(s)_n\wedge S^m\to C(s)_{n+m}
\]
is the smash product of the based map
\[
e_{n,m}\co C_n(\mathfrak{su}_s)^+\to C_{n+m}(\mathfrak{su}_s)^+
\]
sending
\[
(X_1,\dots,X_n) \mapsto (X_1,\dots,X_n,0,\dots,0)
\]
and the canonical homeomorphism $S^n\wedge S^m\cong S^{n+m}$, in formulas
\begin{equation*}
\begin{split}
& [(X_1,\dots,X_n)\wedge x \wedge f]\wedge  y \\& \quad\quad  \mapsto [(X_1,\dots,X_n,0,\dots,0)\wedge (x \wedge y) \wedge \psi_{n,m}\circ (f\otimes j_0)]\,,
\end{split}
\end{equation*}
where $j_0\co \C\to \mathrm{Sym}(\C^m)$ is the inclusion of the $\C\cdot 1$ summand.
\end{construction}

Note that $C(s)_0=\mathrm{pt}$ for all $s\geq 1$.

\begin{proposition} \label{prop:stablequotient}
For every $s\in \N_{>0}$, there is an isomorphism of symmetric spectra
\[
R^s(ku)/R^{s-1}(ku) \cong C(s)\,.
\]
\end{proposition}
\begin{proof}
We define an auxiliary symmetric spectrum $\tilde{C}(s)$ by
\[
\tilde{C}(s)_n=C_n(\mathfrak{u}_s)^+\wedge_{U(s)}L(\C^s,\mathrm{Sym}(\C^n))_+\, .
\]
The structure maps
\[
\tilde{C}(s)_n\wedge S^m\to \tilde{C}(s)_{n+m}
\]
are defined by
\begin{equation*}
\begin{split}
& [(X_1,\dots,X_n)\wedge f]\wedge x_1\wedge \cdots \wedge x_m \\& \quad\quad\quad  \mapsto [(X_1,\dots,X_n,ix_1\mathrm{Id},\dots ix_m \mathrm{Id})\wedge \psi_{n,m}\circ (f\otimes j_0)]\,,
\end{split}
\end{equation*}
where $j_0\co \C\to \mathrm{Sym}(\C^m)$ is the inclusion of the $\C\cdot 1$ summand. It is a simple diagram chase to check that the isomorphism of Theorem \ref{thm:unstablerankfiltrationku},
\[
R^s(ku_n)/R^{s-1}(ku_n)\cong \tilde{C}(s)_n\,,
\]
defines an isomorphism of symmetric spectra $R^s(ku)/R^{s-1}(ku)\cong \tilde{C}(s)$. We leave this verification to the reader.

Let $\mathfrak{z}\subseteq \mathfrak{u}_s$ be the centre. The isomorphism of Lie algebras $\mathfrak{u}_s \cong \mathfrak{su}_s\oplus \mathfrak{z}$, given by
\[
X \mapsto \left(X-\frac{\mathrm{tr}(X)}{s} \mathrm{Id},\frac{\mathrm{tr}(X)}{s} \mathrm{Id}\right)\,,
\]
induces a $(U(s)\times \Sigma_n)$-equivariant homeomorphism
\[
C_n(\mathfrak{u}_s)^+\cong C_n(\mathfrak{su}_s)^+\wedge S^n\, ,
\]
where $\Sigma_n$ acts on $C_n(\mathfrak{su}_s)^+$ by permutation of the components of an $n$-tuple, on $S^n$ through the standard representation, and diagonally on the smash product. This homeomorphism fits into the following commutative diagram of $(U(s)\times \Sigma_n\times \Sigma_m)$-equivariant maps
\[
\xymatrix{
C_n(\mathfrak{u}_s)^+\wedge S^m\ar[r] \ar[d]^-\cong & C_{n+m}(\mathfrak{u_s})^+ \ar[d]^-\cong  \\
C_n(\mathfrak{su}_s)^+\wedge S^n\wedge S^m \ar[r] & C_{n+m}(\mathfrak{su}_s)^+\wedge S^{n+m}
}
\]
where the top horizontal map is given by
\[
(X_1,\dots,X_n)\wedge x_1\wedge \dots \wedge x_m \mapsto (X_1,\dots,X_n,ix_1\mathrm{Id},\dots,ix_m\mathrm{Id})
\]
and the bottom horizontal map is the smash product of $e_{n,m}$ and the canonical homeomorphism $S^n\wedge S^m\cong S^{n+m}$. This shows that there is an isomorphism of symmetric spectra $\tilde{C}(s)\cong C(s)$.
\end{proof}

\begin{remark}
The multiplication maps which make $ku$ a commutative ring spectrum do not restrict to $R^s(ku)$, but they do induce commutative pairings
\[
R^s(ku)\wedge R^t(ku)\to R^{st}(ku)
\]
between the various subspectra of $ku$. Under the isomorphism of Proposition \ref{prop:stablequotient} this gives commutative pairings
\[
C(s)\wedge C(t)\to C(st)
\]
for all $s,t\in \N$. Explicitly, the $(\Sigma_n\times \Sigma_m)$-equivariant component map
\[
C(s)_n\wedge C(t)_m\to C(st)_{n+m}
\]
is given by
\begin{equation*}
\begin{split}
& [(X_1,\dots,X_n)\wedge x\wedge f]\wedge [(Y_1,\dots,Y_m)\wedge y\wedge g] \\& \quad \mapsto [(X_1\otimes \mathrm{Id}_t,\dots,X_n\otimes \mathrm{Id}_t,\mathrm{Id}_s\otimes Y_1,\dots,\mathrm{Id}_s\otimes Y_m)\wedge x\wedge y\wedge \psi_{n,m}\circ (f\otimes g)]\, .
\end{split}
\end{equation*}
\end{remark}

\section{The homotopy type of the subquotients} \label{sec:homotopytype}

We will now identify the spectrum $C(s)$ in the stable homotopy category. To this end, recall that
\[
e_{n,1}\co C_n(\mathfrak{su}_s)^+\to C_{n+1}(\mathfrak{su}_s)^+
\]
is the based map sending $(X_1,\dots,X_n)$ to $(X_1,\dots,X_n,0)$.

\begin{lemma} \label{lem:connectivity}
For all $n\in \N_{>0}$, the map
\begin{equation*}
\begin{split}
e_{n,1}\wedge \psi_{n,1}\circ (id\otimes j_0) \co & C_n(\mathfrak{su}_s)^+\wedge_{U(s)} L(\C^s,\mathrm{Sym}(\C^n))_+\\&\quad \to C_{n+1}(\mathfrak{su}_s)^+\wedge_{U(s)} L(\C^s,\mathrm{Sym}(\C^{n+1}))_+
\end{split}
\end{equation*}
is $(n-1)$-connected.
\end{lemma}
\begin{proof}
If $s=1$, then $C_n(\mathfrak{su}_s)^+=S^0$ and the map in question is a homotopy equivalence. Thus, we may assume $s\geq 2$. To simplify the notation, let us write
\[
A_n=C_n(\mathfrak{su}_s)^+\wedge_{U(s)} L(\C^s,\mathrm{Sym}(\C^n))_+\, .
\]
Since $s\geq 2$, $C_n(\mathfrak{su}_s)^+$ is simply-connected by \cite[Proposition A.6]{AGG}, and hence $A_n$ is simply-connected as well. By the relative Hurewicz theorem, it suffices to show that $H_i(A_{n+1},A_n;\Z)=0$ for all $i\leq n-1$.

The spectrum $R^s(ku)$ arises from a special $\Gamma$-space and is therefore an $\Omega$-spectrum above level $0$, and connective. In particular, $R^s(ku_n)$ is $(n-1)$-connected for all $n$. By Freudenthal's suspension theorem, the structure maps $R^s(ku_n)\wedge S^1\to R^s(ku_{n+1})$ are then $2n$-connected. The structure map
\[
C(s)_n \wedge S^1=\Sigma^{n+1} A_n  \xrightarrow{\Sigma^{n+1}(e_{n,1}\wedge \psi_{n,1}\circ (id\otimes j_0))} C(s)_{n+1}=\Sigma^{n+1} A_{n+1}
\]
(see Construction \ref{cons:quotientspectrum}) arises by taking cofibres. By the five lemma we have
\[
H_i(\Sigma^{n+1}A_{n+1},\Sigma^{n+1}A_n;\Z)=0
\]
for all $i\leq 2n$. The suspension isomorphism implies that $H_i(A_{n+1},A_n;\Z)=0$ for all $i\leq n-1$, as desired.
\end{proof}

By taking the colimit over the maps $e_{n,1}\co C_n(\mathfrak{su}_s)^+\to C_{n+1}(\mathfrak{su}_s)^+$ we define the space
\[
C_\infty(\mathfrak{su}_s)^+=\mathrm{colim}_{n\in \N}\, C_n(\mathfrak{su}_s)^+\,.
\]
We note that classical equivariant triangulation theorems can be used to show that for every $n\in \N$ the space $C_n(\mathfrak{su}_s)^+$ can be given a $U(s)$-CW-structure in such a way that $C_{n}(\mathfrak{su}_s)^+$ is a $U(s)$-subcomplex of $C_{n+1}(\mathfrak{su}_s)^+$ (see \cite[Section 3.4]{AGG} for similar statements). In particular, the colimit above is also a homotopy colimit.

The following proposition along with Proposition \ref{prop:stablequotient} proves Theorem \ref{thm:intro2} from the introduction.

\begin{proposition} \label{prop:homotopytype}
There is an isomorphism in the stable homotopy category
\[
C(s) \cong  \Sigma^\infty\left( C_\infty(\mathfrak{su}_s)^+\wedge_{U(s)} EU(s)_+\right)\, .
\]
\end{proposition}
\begin{proof}
The spectra $R^{s-1}(ku)$ and $R^s(ku)$ are semi-stable (since they come from orthogonal spectra), hence so is their cofibre $C(s)$. Since $C(s)$ is semi-stable, its homotopy type is determined by its underlying pre-spectrum. It is clear that as a pre-spectrum $C(s)$ is level-equivalent to the spectrum with $n$-th space ($n\geq 1$)
\[
X(s)_n=(C_n(\mathfrak{su}_s)^+\wedge S^n)\wedge_{U(s)} L(\C^s,\mathrm{Sym}(\C))_+
\]
and structure maps given by the smash product of $e_{n,1}$ and the homeomorphism $S^n\wedge S^1\cong S^{n+1}$. In the spectrum $X(s)$ the universe $\mathrm{Sym}(\C)$ no longer changes with the levels, so we write $L(\C^s,\mathrm{Sym}(\C))=EU(s)$. The $0$-th space is $X(s)_0=S^0$.

Define the map of pre-spectra
\[
X(s) \to  \Sigma^\infty\left( C_\infty(\mathfrak{su}_s)^+ \wedge_{U(s)} EU(s)_+ \right)
\]
to be the identity in level $0$ and the map
\[
\Sigma^n (C_n(\mathfrak{su}_s)^+\wedge_{U(s)}EU(s)_+)\to \Sigma^n(C_\infty(\mathfrak{su}_s)^+ \wedge_{U(s)} EU(s)_+)
\]
induced by the inclusion $C_n(\mathfrak{su}_s)^+\to C_\infty(\mathfrak{su}_s)^+=\mathrm{colim}_{n\in \N}\, C_n(\mathfrak{su}_s)^+$ in all levels $n\geq 1$. This map of pre-spectra induces an isomorphism on all homotopy groups by Lemma \ref{lem:connectivity}. It is therefore an isomorphism in the stable homotopy category.
\end{proof}

\begin{example}[Relationship with the complete unordered flag manifold] \label{ex:flag}
When $s=1$, then $C_\infty(\mathfrak{su}_s)^+=S^0$, so
\[
C(1)\cong \Sigma^\infty \left(BU(1)_+\right)\, .
\]
Suppose now that $s=p$ is a prime (even or odd). Let $U(1)^p\subseteq U(p)$ be a maximal torus and let $H=U(1)^p\cap SU(p)$ be a maximal torus for $SU(p)$. Let $\mathfrak{h}$ be the Lie algebra of $H$. As a $\Sigma_p$-representation, $\mathfrak{h}$ is the $(p-1)$-dimensional reduced standard representation. By fixing an Ad-invariant norm $\norm{\cdot}$ on $\mathfrak{su}_p$, we define the space of commuting $n$-tuples of unit norm:\footnote{This space would correspond to $\overline{\mathcal{L}_{|\cdot|=1}(s,S^n)}$ in \cite[Section 3.1]{Hausmann}.}
\[
C_n^1(\mathfrak{su}_p)=\{(X_1,\dots,X_n)\in C_n(\mathfrak{su}_p)\mid \textstyle{\sum_{i=1}^n \norm{X_i}^2=1}\}\, .
\]
Let $S(\mathfrak{h}^n)\subseteq \mathfrak{h}^n$ be the unit sphere and consider the $U(p)$-equivariant map
\[
\phi_n \co U(p)/U(1)^p \times_{\Sigma_p} S(\mathfrak{h}^n) \to C_n^1(\mathfrak{su}_p)
\]
defined by
\[
[(gU(1)^p,(X_1,\dots,X_n))] \mapsto (\mathrm{Ad}_g(X_1),\dots,\mathrm{Ad}_g(X_n))\,,
\]
where $\mathrm{Ad}_g\co \mathfrak{su}_p\to \mathfrak{su}_p$ is the adjoint action. By \cite[Lemma A.13]{AGG}, the map $\phi_n$ is a mod-$p$ homology isomorphism for all $n$ (if $p=2$, then $\phi_n$ is even a homeomorphism). Note that $U(p)/U(1)^p\times_{\Sigma_p} S(\mathfrak{h}^n)$ is the total space of a sphere bundle over the manifold of complete unordered flags in $\C^p$, i.e., over
\[
\overline{\mathrm{Fl}}(\C^p):=(U(p)/(\Sigma_p \wr U(1))\, .
\]
Let $C^1_n(\mathfrak{su}_p)^\lozenge$ denote the unreduced suspension of $C^1_n(\mathfrak{su}_p)$ with basepoint the cone point at $1$. Then $C^1_n(\mathfrak{su}_p)^\lozenge\cong C_n(\mathfrak{su}_p)^+$ as $U(p)$-spaces (see \cite[Lemma A.5]{AGG}), so that in the colimit $n\to \infty$ there is a mod-$p$ homology isomorphism
\[
\overline{\mathrm{Fl}}(\C^p)^\lozenge \wedge_{U(p)} EU(p)_+\to C_\infty(\mathfrak{su}_p)^+\wedge_{U(p)} EU(p)_+\, .
\]
Hence, there is an equivalence mod-$p$
\[
R^p(ku)/R^{p-1}(ku)=C(p)\simeq_{p} \Sigma^\infty \left(\overline{\mathrm{Fl}}(\C^p)^\lozenge \wedge_{U(p)} EU(p)_+\right)\, .
\]
\end{example}

We will now show how Theorem \ref{thm:intro2} reduces to a well-known description of the subquotients of the stable rank filtration of $ku$. Let $\mathcal{C}$ be a collection of subgroups of a compact Lie group $G$, i.e., a set of subgroups that is closed under conjugation. Recall that a $G$-space $X$ is a \emph{universal space} for the collection $\mathcal{C}$ if
\begin{itemize}
\item the isotropy groups of the $G$-action on $X$ belong to $\mathcal{C}$
\item for each $K\in \mathcal{C}$ the fixed point space $X^K$ is contractible.
\end{itemize}
A universal space is determined up to $G$-equivalence by these two properties. We write $X=E\mathcal{C}$ if $X$ is universal for $\mathcal{C}$.

Fix $s\in \N_{>0}$ and let $\mathcal{L}_s$ be the collection of \emph{complete} subgroups of $U(s)$, i.e., the subgroups which are conjugate to one of the form
\[
U(n_1)\times \cdots \times U(n_k)
\]
with $k>1$ and $\sum_i n_i=s$. See \cite[Section 9]{ALfiltered} for more details.

\begin{proposition}[{Cf. \cite[Prop. 3.6]{Hausmann}}] \label{prop:universal1}
The $U(s)$-space $C^1_\infty(\mathfrak{su}_s)$ is universal for the collection $\mathcal{L}_s$.
\end{proposition}
\begin{proof}
We first show that the isotropy groups of $C_\infty^1(\mathfrak{su}_s)$ belong to $\mathcal{L}_s$. An element of $C_\infty^1(\mathfrak{su}_s)$ is represented by a tuple $(X_1,\dots,X_n)\in C_n^1(\mathfrak{su}_s)$ for some $n\geq 1$.  Since the $X_i$'s commute pairwise, they give a simultaneous eigenspace decomposition of $\C^s$. This is the unique coarsest orthogonal decomposition $\C^s=\bigoplus_{j=1}^k V_j$ such that $X_i$ acts on $V_j$ by a scalar $x_{ij}\in i\R$ for all $i,j$. Elementary linear algebra shows that the isotropy group of $(X_1,\dots,X_n)$ is precisely the subgroup of $U(s)$ which leaves invariant the subspaces $V_1,\dots,V_k$. But this subgroup is conjugate to $U(n_1)\times \cdots \times U(n_k)$ for $n_j=\dim_{\C}(V_j)$. Note that $k>1$, because otherwise the matrices $X_i$ being traceless and diagonal would have to be zero for all $i$. This would contradict the fact that $(X_1,\dots,X_n)$ has unit norm.

To see that for each $K\in \mathcal{L}_s$ the space $C_\infty^1(\mathfrak{su}_s)^K$ is contractible, we may assume that $K=U(n_1)\times \cdots \times U(n_k)$. Let $\mathfrak{h}\subseteq \mathfrak{su}_s$ be the Lie subalgebra consisting of the diagonal matrices in $\mathfrak{su}_s$ only. If $(X_1,\dots,X_n)\in C^1_n(\mathfrak{su}_s)$ is fixed by $K$, then each $X_i$ must be a traceless block diagonal matrix with blocks of the form $x_{ij} \mathrm{Id}_{n_j}$, $j=1,\dots,k$, $x_{ij}\in i\R$. In particular, $(X_1,\dots,X_n)\in S(\mathfrak{h}^n)$. This shows that $C_n^1(\mathfrak{su}_s)^K$ is the intersection of $S(\mathfrak{h}^n)$ with finitely many hyperplanes in $\mathfrak{h}^n$ enforcing the vanishing of the trace as well as the equality of the first $n_1$ entries of $X_i$, the next $n_2$ entries of $X_i$ and so on -  for all $i$. As $n$ tends to infinity, this intersection is a sphere $S^\infty$, hence contractible.
\end{proof}

The proposition shows that $C_\infty^1(\mathfrak{su}_s)=E\mathcal{L}_s$, and hence that $C_\infty(\mathfrak{su}_s)^+=E\mathcal{L}_s^\lozenge$.

The collection $\mathcal{L}_s$ can also be viewed as a topological poset ordered by inclusion of subgroups. The geometric realisation $|\mathcal{L}_s|$ is the complex of proper direct sum decompositions of $\C^s$ introduced by Arone \cite{AroneTop}. There is a $U(s)$-equivariant map $E\mathcal{L}_s\to |\mathcal{L}_s|$ which is a homotopy equivalence (see \cite[§2.2]{AL} and the references therein). As a result, Theorem \ref{thm:intro2} becomes the following theorem of Arone and Lesh  (see \cite[§2.2]{AL}).

\begin{corollary}[\cite{AL}]
There are isomorphisms in the stable homotopy category
\begin{equation*}
\begin{split}
R^s(ku)/R^{s-1}(ku)& \cong \Sigma^\infty \left(E\mathcal{L}_s^\lozenge \wedge_{U(s)} EU(s)_+\right) \\& \cong \Sigma^\infty\left(|\mathcal{L}_s|^\lozenge \wedge_{U(s)} EU(s)_+\right)\, .
\end{split}
\end{equation*}
\end{corollary}

As a by-product, we remark that the mod-$p$ homology equivalence
\[
\overline{\mathrm{Fl}}(\C^p)\simeq_p C^1_\infty(\mathfrak{su}_p) \simeq |\mathcal{L}_p|
\]
(see Example \ref{ex:flag}) reduces the computation of $H^\ast(\overline{\mathrm{Fl}}(\C^p);\F_p)$ undertaken in \cite{JanaGuerra} to the calculation of $H^\ast(|\mathcal{L}_p|;\F_p)$. The latter is known thanks to the work of Arone.

\begin{corollary} \label{cor:flag} Let $p$ be an odd prime and let $\overline{\mathrm{Fl}}(\C^p)$ be the manifold of complete unordered flags in $\C^p$. Let $\mathcal{A}(0)=\Lambda(\beta)$ be the subalgebra of the mod-$p$ Steenrod algebra generated by the Bockstein homomorphism. Then there is an isomorphism of $\mathcal{A}(0)$-modules
\[
\tilde{H}^\ast(\overline{\mathrm{Fl}}(\C^p);\F_p) \cong \Sigma^{2p-3}(\mathcal{A}(0)\otimes \Lambda(x_1,\dots,x_{p-2}))\,,
\] 
where $|x_{i}|=2i-1$ and the suspension indicates a shift so that the smallest degree with non-trivial cohomology is $2p-3$.
\end{corollary}
\begin{proof}
By \cite[Theorem 10.1]{ALfiltered}, there is a mod-$p$ equivalence
\[
\Sigma^\infty |\mathcal{L}_s|^\lozenge \simeq_p \Sigma \mathrm{Map}(|\mathcal{L}_s|^\lozenge,\Sigma^{\infty}S^{\mathfrak{u}_p})\,,
\]
and the cohomology of the right hand side is displayed in \cite[Theorem 4(b)]{AroneTop}.
\end{proof}

It follows that the mod-$p$ Poincar{\'e} polynomial of $\overline{\mathrm{Fl}}(\C^p)$ is
\[
P(t)=1+t^{2p-3}(1+t)\prod_{i=1}^{p-2}(1+t^{2i-1})\,,
\]
which agrees with \cite[Corollary 6.9]{JanaGuerra}.

\begin{remark}
There is an even more direct way of showing that $\overline{\mathrm{Fl}}(\C^p)^\lozenge \simeq_p |\mathcal{L}_p|^\lozenge$. As explained in the proof of \cite[Proposition 9.1]{AroneCW}, $|\mathcal{L}_m|^\lozenge$ can be identified with the total homotopy cofibre of a certain $(m-1)$-dimensional cubical diagram $\mathcal{X}$ indexed by subsets $U\subseteq \{2,\dots,m\}$. The values $\mathcal{X}(U)$ are disjoint unions of $U(m)$-orbits $U(m)/H$, where $H\subseteq U(m)$ is a product of iterated wreath products of symmetric groups with standard subgroups $1\times \cdots \times 1\times U(k)\times 1\times \cdots \times 1\subseteq U(m)$. If $m=p$ and $U\neq \{p\}$, the symmetric groups that appear have order prime to $p$. Consequently, for all $W\neq \emptyset$ with $p\not\in W$, the maps $\mathcal{X}(\{p\}\cup W)\to \mathcal{X}(W)$ which are shown in \cite{AroneCW} to be rational homology isomorphisms are, in fact, homology isomorphisms with coefficients in $\F_p$. It follows that the total homotopy cofibre of $\mathcal{X}$ is mod-$p$ equivalent with the homotopy cofibre of the map $\mathcal{X}(\{p\})=U(p)/(\Sigma_p\wr U(1))\to \mathcal{X}(\emptyset)=\mathrm{pt}$, i.e., with $\overline{\mathrm{Fl}}(\C^p)^\lozenge$.
\end{remark}

\section{Statements for real $K$-theory} \label{sec:realvariants}

In this section we summarise the real variants of the foregoing results. The commutative symmetric ring spectrum $ko$ is defined in the same fashion as $ku$, but using real subspaces of real inner product spaces everywhere, see \cite[Remark 6.3.12]{Global}. Let $\mathcal{U}$ be a real inner product space of finite or countably infinite dimension. Then $ko(-,\mathcal{U})$ denotes the $\Gamma$-space which assigns to $\langle k\rangle\in \Fin_\ast$ the space of $k$-tuples of mutually orthogonal finite dimensional subspaces of $\mathcal{U}$. The $n$-th space of $ko$ is defined by
\[
ko_n=ko(S^n,\mathrm{Sym}(\R^n))\,,
\]
with $\Sigma_n$-action and structure maps as in the complex case.

Let $\mathcal{U}_{\C}$ be the complexification of $\mathcal{U}$. Let $U(\mathcal{U}_\C)^{\mathrm{sym}}\subseteq U(\mathcal{U}_\C)$ denote the subspace of those unitary transformations whose eigenspaces are the complexification of real subspaces of $\mathcal{U}$. Equivalently, these are the unitary transformations which are represented by symmetric matrices with respect to an orthonormal basis of $\mathcal{U}$. Define
\[
\overline{\mathscr{C}}_n^\mathrm{sym} =\{[(A_1,\dots,A_n)]\in \overline{\mathscr{C}}_n  \mid A_1,\dots,A_n\in U(\mathrm{Sym}(\C^n))^{\mathrm{sym}} \}\, .
\]

The following is then the real analogue of Lemma \ref{lem:model}.

\begin{lemma} \label{lem:modelreal}
For all $n\in \N$, there is a $\Sigma_n$-equivariant homeomorphism
\[
ko_n\cong \overline{\mathscr{C}}_n^\mathrm{sym}  \, .
\]
\end{lemma}
\begin{proof}
Complexification identifies $ko_n$ with a subspace of $ku_n$ (see \cite[Remark 6.3.12]{Global}). Then the homeomorphism in question is just the restriction of the homeomorphism in Lemma \ref{lem:model}.
\end{proof}

For every $n\in \N$, we define a rank filtration of $ko_n$ by letting $R^s(ko_n)\subseteq ko_n$ be the subspace of configurations $[(V_1,x_1),\dots,(V_k,x_k)]\in ko_n$ satisfying
\[
\sum_{\substack{i=1\\ x_i\neq \textnormal{ basepoint}}}^k \dim_{\R}(V_i) \leq s\, .
\]
Let $\mathfrak{sym}_s\subseteq \R^{s^2}$ be the space of $s\times s$ real symmetric matrices, and let
\[
C_n(\mathfrak{sym}_s)=\{(X_1,\dots,X_n)\in \mathfrak{sym}_s^n\mid [X_i,X_j]=0\textnormal{ for all }1\leq i,j\leq n\}
\]
be the space of $n$-tuples of pairwise commuting such matrices. The orthogonal group $O(m)$ acts on $C_n(\mathfrak{sym}_m)$ by simultaneous conjugation and this extends to an action on the one-point compactification $C_n(\mathfrak{sym}_s)^+$.

The real analogue of Theorem \ref{thm:unstablerankfiltrationku} then reads:
\begin{theorem} \label{thm:unstablerankfiltrationko}
For $s\in \N$ and $n\in \N_{>0}$, there is a $\Sigma_n$-equivariant homeomorphism
\[
R^s(ko_n)/R^{s-1}(ko_{n})\cong C_n(\mathfrak{sym}_s)^+\wedge_{O(n)} L(\R^s,\mathrm{Sym}(\R^n))_+\, .
\]
\end{theorem}
In particular, this gives Theorem \ref{thm:intro3} from the introduction.
\begin{proof}
The key fact is that the Cayley transform $\psi_s\co \mathfrak{u}_s\to U(s)$ is equivariant with respect to transposition, which is evident from the formula displayed in (\ref{eq:cayley}). Since a skew-Hermitian matrix is symmetric if and only if it is of the form $iX$ for a real symmetric matrix $X$, the Cayley transform restricts to an $O(s)$-equivariant map
\[
\mathfrak{sym}_s \to U(s)^{\mathrm{sym}}
\]
sending $X \mapsto \psi_s(iX)$. This map is a homeomorphism onto the subspace of those $A\in U(s)^{\mathrm{sym}}$  such that $A-\mathrm{Id}$ is non-singular. The rest of the proof is entirely parallel to the proof of Theorem \ref{thm:unstablerankfiltrationku}.
\end{proof}

There is an $O(s)$-equivariant decomposition $\mathfrak{sym}_s\cong \overline{\mathfrak{sym}}_s\oplus \R$, where $ \overline{\mathfrak{sym}}_s$ is the subspace of traceless matrices on which $O(s)$ acts by conjugation, and $\R$ represents the scalar matrices with trivial $O(s)$-action. This decomposition induces a $(O(s)\times \Sigma_n)$-equivariant homeomorphism
\[
C_n(\mathfrak{sym}_s)^+\cong C_n(\overline{\mathfrak{sym}}_s)^+\wedge S^n\, .
\]
This yields a description of the symmetric spectrum $R^s(ko)/R^{s-1}(ko)$ along the lines of Construction \ref{cons:quotientspectrum} using $C_n(\overline{\mathfrak{sym}}_s)^+$ instead of $C_n(\mathfrak{su}_s)^+$. In particular, we have the following variant of Theorem \ref{thm:intro2}.
\begin{theorem} \label{thm:homotopytypereal}
There is an isomorphism in the stable homotopy category
\[
R^s(ko)/R^{s-1}(ko)\cong \Sigma^\infty\left(C_\infty(\overline{\mathfrak{sym}}_s)^+ \wedge_{O(s)} EO(s)_+\right)\, .
\]
\end{theorem}
\begin{proof}
The proof is analogous to that of Theorem \ref{thm:intro2}. To establish the analogue of Lemma \ref{lem:connectivity} we must see that $C_n(\overline{\mathfrak{sym}}_s)^+$ is simply-connected if $s\geq 2$. Fix an $O(s)$-invariant inner product on $\overline{\mathfrak{sym}}_s$ and let $C_n^1(\overline{\mathfrak{sym}}_s)$ denote the subspace of commuting $n$-tuples of unit norm. Since $C_n(\overline{\mathfrak{sym}}_s)^+\cong C_n^1(\overline{\mathfrak{sym}}_s)^\lozenge$, it suffices to see that $C_n^1(\overline{\mathfrak{sym}}_s)$ is path-connected. When $s=2$ it is elementary to see that $C_n^1(\overline{\mathfrak{sym}}_s)$ is the unit sphere bundle in the $n$-fold sum of the M{\"o}bius bundle over $\R P^1$, thus path-connected. Now assume $s\geq 3$. Let $\overline{\rho}_s\subseteq \overline{\mathfrak{sym}}_s$ denote the subspace of diagonal matrices and let $S(n\overline{\rho}_s)$ be the unit sphere in $\overline{\rho}_s^n\cong \R^{n(s-1)}$ (as a $\Sigma_s$-space, $\overline{\rho}_s$ is the reduced standard representation, hence the notation). Since the map
\begin{equation*}
\begin{split}
SO(s) \times  S(n\overline{\rho}_s) &\to C_n^1(\overline{\mathfrak{sym}}_s) \\
(A,X_1,\dots,X_n) & \mapsto (AX_1A^{-1},\dots,AX_nA^{-1})
\end{split}
\end{equation*}
is surjective (by the spectral theorem) and the domain is path-connected, the space $C_n^1(\overline{\mathfrak{sym}}_s)$ is path-connected.
\end{proof}

Finally, we have the following relationship of $C_\infty^1(\overline{\mathfrak{sym}}_s)$ with the complex of proper direct-sum decompositions of $\R^s$. Let $\mathcal{L}_s^{\R}$ denote the collection of complete subgroups of $O(s)$, i.e., subgroups that are conjugate to $O(n_1)\times \cdots \times O(n_k)$ for some $k>1$ and $\sum_{j} n_j=s$. As in Proposition \ref{prop:universal1} one shows:

\begin{proposition} \label{prop:universal2}
The $O(s)$-space $C_\infty^1(\overline{\mathfrak{sym}}_s)$ is universal for the collection $\mathcal{L}_s^{\R}$.
\end{proposition}

The proposition implies that $C_\infty^1(\overline{\mathfrak{sym}}_s)=E\mathcal{L}_s^{\R}\simeq |\mathcal{L}_s^{\R}|$. Unlike the complex case, however, this does not lead to a computation of the mod-$p$ cohomology of the complete unordered flag manifold $\overline{\mathrm{Fl}}(\R^p)=O(p)/(\Sigma_p\wr O(1))$.

\end{document}